\documentclass[graybox]{svmult}

\usepackage{type1cm}        
\usepackage{makeidx}         
\usepackage{graphicx}        
\usepackage{multicol}        
\usepackage[bottom]{footmisc}

\usepackage[colorlinks=true]{hyperref}
\usepackage{color}

\usepackage{newtxtext}       %
\usepackage{newtxmath}       

\makeindex             

\begin{document}

\title*{Coefficient-Robust A Posteriori Error Estimation for H(curl)-elliptic Problems}
\author{Yuwen Li}
\institute{Yuwen Li \at The Pennsylvania State University, University Park, \email{yuwenli925@gmail.com}
}
%
%
\maketitle

\abstract{We extend the framework of a posteriori error estimation by preconditioning in [Li, Y., Zikatanov, L.: Computers \& Mathematics with Applications. \textbf{91}, 192-201 (2021)] and derive new a posteriori error estimates for H(curl)-elliptic two-phase interface problems. The proposed error estimator provides two-sided bounds for the discretization error and is robust with respect to coefficient variation under mild assumptions. For H(curl) problems with constant coefficients, the performance of this estimator is numerically compared with the one analyzed in [Sch\"oberl, J.: Math.~Comp. \textbf{77}(262), 633-649 (2008)]. }

\section{Introduction}
\label{sec:1}
Adaptive mesh refinement (AMR) is a popular tool in numerical simulations as it is able to resolve singularity from nonsmooth data and irregular space domains. A building block of AMR is a posteriori error estimation, see, e.g., \cite{Verfurth2013} for a classical introduction. On the other hand, preconditioners are discrete operators used  to accelerate Krylov subspace methods for solving sparse linear systems (cf.~\cite{Xu1992}). Recently, \cite{LiZikatanov2020CAMWA,LiZikatanov2020arXiv} introduced a novel framework linking posteriori error estimation and \emph{preconditioning} in the Hilbert space. Such an approach yields many old and new error estimators for boundary value problems posed on de Rham complexes. 

In particular, for the positive-definite H(curl) problem,  \cite{LiZikatanov2020arXiv} presents a new residual estimator robust w.r.t.~high-contrast constant coefficients. 
In this paper, we extend the idea in \cite{LiZikatanov2020arXiv} to the H(curl) interface problem and derive new a posteriori error estimates robust w.r.t.~\emph{both} extreme coefficient magnitude as well as large  coefficient jump. The analysis avoids regularity assumptions used in existing works. We  numerically compare the performance of the estimator in \cite{LiZikatanov2020arXiv} with the one analyzed in \cite{Schoberl2008}.

\subsection{H(curl)-Elliptic Problems}\label{secpre}
Let $\Omega\subset\mathbb{R}^d$ with $d\in\{2,3\}$ be a bounded Lipschitz domain, and $n$ be a unit vector normal to $\partial\Omega$. Let $\nabla\times$ be the usual curl in $\mathbb{R}^3$, $\nabla\times=(\partial_{x_2},-\partial_{x_1})\cdot$ in  $\mathbb{R}^2$. We define
\begin{equation*}
    V=\left\{v\in [L^2(\Omega)]^d: \nabla\times v\in [L^2(\Omega)]^{\frac{d(d-1)}{2}},~v\wedge n=0\text{ on }\partial\Omega\right\},
\end{equation*}
where $v\wedge n=v\times n$ in $\mathbb{R}^3$, $v\wedge n=v\cdot n^\perp$ in $\mathbb{R}^2$ with $n^\perp$ the counter-clockwise rotation of $n$ by $\frac{\pi}{2}$, and $[X]^d$ the Cartesian product of $d$ copies of $X$. Let $(\cdot,\cdot)_{\Omega_0}$ denote the $L^2(\Omega_0)$ inner product and $(\cdot,\cdot)=(\cdot,\cdot)_{\Omega}.$
Given $f\in L^2(\Omega)$ and positive $\varepsilon, \kappa\in L^\infty(\Omega)$, the H(curl)-elliptic boundary value problem seeks $u\in V$ s.t.
\begin{equation}\label{Maxwell}
(\varepsilon\nabla\times u,\nabla\times v)\text{ + }(\kappa u,v)=(f,v),\quad\forall v\in V.
\end{equation}
The space $V$ is equipped with the $V$-norm and energy inner product based on
\begin{equation*}
    (v,w)_V=(\varepsilon \nabla\times v,\nabla\times w)\text{ + }(\kappa v,w),\quad\forall v, w\in V.
\end{equation*}

Let $\mathcal{T}_h$ be a conforming tetrahedral or hexahedral partition of $\Omega.$ Problem \eqref{Maxwell} is often discretized using the N\'ed\'elec edge element space $V_h\subset V$.
The discrete problem is to find $u_h\in V_h$ s.t.
\begin{equation}\label{disMaxwell}
(\varepsilon\nabla\times u_h,\nabla\times v)\text{ + }(\kappa u_h,v)=(f,v),\quad\forall v\in V_h.
\end{equation}

The semi-discrete Maxwell equation is an important example of  \eqref{Maxwell}. In this case, $\varepsilon$ is the reciprocal of the magnetic permeability and $\kappa$ is proportional to $1/\tau^2$, where $\tau$ is the time stepsize. Therefore, we are interested in $\varepsilon$ with large jump and potentially huge $\kappa$. In particular, we assume $\kappa>0$ is a constant, {$\Omega_1\subset\Omega$, $\Omega_2\subset\Omega$ are non-overlapping and simply-connected polyhedrons aligned with $\mathcal{T}_h$,}  $\bar{\Omega}=\bar{\Omega}_1\cup\bar{\Omega}_2$, and
\begin{equation}
    \varepsilon|_{\Omega_1}=\varepsilon_1,\quad\varepsilon|_{\Omega_2}=\varepsilon_2,
\end{equation}
where $\varepsilon_1\geq\varepsilon_2>0$ are constants. The interface is $\Gamma:=\bar{\Omega}_1\cap\bar{\Omega}_2.$ A posteriori error analysis for more general $\varepsilon, \kappa$ is beyond the scope of this work but is possible by making monotonicity-type assumptions on distributions of $\varepsilon$ and $\kappa,$ cf.~\cite{BernardiVerfurth2000,CaiCao2015}.

Throughout the rest of this paper, we say $\alpha\preccurlyeq \beta$ provided $\alpha\leq C\beta$, where $C$ is an absolute constant depending solely on $\Omega$, the aspect ratio of elements in $\mathcal{T}_h,$ and the polynomial degree used in $V_h.$ We say $\alpha\simeq \beta$ if $\alpha\preccurlyeq \beta$ and $\beta\preccurlyeq \alpha.$ Given a Lipschitz manifold $\Sigma\subset\Omega$, by $\|\cdot\|_\Sigma$ we denote the $L^2(\Sigma)$ norm.

\section{Nodal Auxiliary Space Preconditioning}\label{secmain}
The key idea in \cite{LiZikatanov2020arXiv} is \emph{nodal auxiliary space preconditioning}, originally proposed in  \cite{HiptmairXu2007} for solving discrete H(curl) and H(div) problems. The auxiliary $H^1$ space here is 
\begin{equation*}
    W=\left\{w\in L^2(\Omega): \nabla w\in [L^2(\Omega)]^d,~w|_{\partial\Omega}=0\right\},
\end{equation*}
endowed with the inner product 
\begin{equation*}
    (w_1,w_2)_W=(\varepsilon \nabla w_1,\nabla w_2)\text{ + }(\kappa w_1,w_2)
\end{equation*}
and the induced $W$-norm. The next regular decomposition (with \emph{mixed} boundary condition, cf.~\cite{LiZikatanov2020arXiv,HiptmairXu2007}) is widely used in the analysis of H(curl) problems.
\begin{theorem}\label{regular0}
Given $v\in V|_{\Omega_1}$, there exist $\varphi\in W|_{\Omega_1}$,  $z\in[W|_{\Omega_1}]^d,$ s.t.~$v=\nabla\varphi\text{ + }z$, 
\begin{align*}
    &\|z\|_{H^1(\Omega_1)}\leq C_0\|\nabla\times v\|,\\
    &\|\varphi\|_{H^1(\Omega_1)}\leq C_0 (\|v\|\text{ + }\|\nabla\times v\|),
\end{align*}
where $C_0$ is a constant depending only on $\Omega_1$.
\end{theorem}

To derive a posteriori error bounds for \eqref{disMaxwell} uniform w.r.t.~constant $\varepsilon\ll\kappa$, the work \cite{LiZikatanov2020arXiv} utilizes the following modified regular decomposition.
\begin{theorem}\label{regular1}
Given $v\in V|_{\Omega_1}$, there exist $\varphi\in W|_{\Omega_1}$, $z\in[W|_{\Omega_1}]^d,$ s.t.~$v=\nabla\varphi\text{ + }z$ and
\begin{align*}
    &\|\varphi\|_{H^1(\Omega_1)}\text{ + }\|z\|\leq C_1\|v\|,\\
    &|z|_{H^1(\Omega_1)}\leq C_1 (\|v\|\text{ + }\|\nabla\times v\|),
\end{align*}
where $C_1$ is a constant depending only on $\Omega_1.$
\end{theorem}

In the following, we give a new regular decomposition robust w.r.t.~constant $\kappa$ and piecewise constant $\varepsilon$. See also \cite{HuShuZou2013} for a weighted Helmholtz decomposition.
\begin{theorem}\label{regular2}
Given $v\in V$, there exist $\varphi\in W$ and $z\in[W]^d,$ s.t.~$v=\nabla\varphi\text{ + }z$ and
\begin{align*}
    &\|\kappa^\frac{1}{2}\varphi\|_{H^1(\Omega)}\text{ + }\|z\|_W\leq C_2\|v\|_V,
\end{align*}
where $C_2$ is a constant depending solely on $\Omega$, $\Omega_1$, $\Omega_2$.
\end{theorem}
\begin{proof}
The proof is divided into two cases. When $\varepsilon_1\geq\kappa$, we use Theorem \ref{regular0} on $\Omega_1$ to obtain $\varphi_1\in H^1(\Omega_1)$, $z_1\in [H^1(\Omega_1)]^d$ both vanishing on $\partial\Omega_1\backslash\Gamma$ s.t.
\begin{equation}\label{case1}
\begin{aligned}
    &v|_{\Omega_1}=\nabla\varphi_1\text{ + }z_1,\\
    &\|z_1\|_{H^1(\Omega_1)}\preccurlyeq\|\nabla\times v\|_{\Omega_1},\\
    &\|\varphi_1\|_{H^1(\Omega_1)}\preccurlyeq\|v\|_{\Omega_1}\text{ + }\|\nabla\times v\|_{\Omega_1}.
\end{aligned}
\end{equation}
When $\varepsilon_1<\kappa$, applying Theorem \ref{regular1} to $v|_{\Omega_1}$ yields $\varphi_1\in H^1(\Omega_1)$, $z_1\in [H^1(\Omega_1)]^d$ s.t.
\begin{equation}\label{case2}
\begin{aligned}
    &v|_{\Omega_1}=\nabla\varphi_1\text{ + }z_1,\quad\varphi_1|_{\partial\Omega_1\backslash\Gamma}=0,~z_1|_{\partial\Omega_1\backslash\Gamma}=0,\\
    &\|\varphi_1\|_{H^1(\Omega_1)}\text{ + }\|z_1\|_{\Omega_1}\preccurlyeq\|v\|_{\Omega_1},\\
    &|z_1|_{H^1(\Omega_1)}\preccurlyeq\|v\|_{\Omega_1}\text{ + }\|\nabla\times v\|_{\Omega_1},
\end{aligned}
\end{equation}
In either case, it holds that
\begin{equation}\label{Omega1}
    \|\kappa^\frac{1}{2}\varphi_1\|_{H^1(\Omega_1)}\text{ + }\|z_1\|_{W|_{\Omega_1}}\preccurlyeq\|v\|_{V|_{\Omega_1}}.
\end{equation}
First let $\hat{\varphi}_1\in H^1(\mathbb{R}^d\backslash\Omega_2)$ and $\hat{z}_1\in[H^1(\mathbb{R}^d\backslash\Omega_2)]^d$ be zero extensions of $\varphi_1$ and $z_1$ to $\mathbb{R}^d\backslash\Omega_2,$ respectively.
Then we take  $\tilde{\varphi}_1\in H^1(\Omega)$, $\tilde{z}_1\in H^1(\Omega)$ to be the Stein universal extensions of $\hat{\varphi}_1$, $\hat{z}_1$ to $\mathbb{R}^d$ satisfying
\begin{equation}\label{extension}
\begin{aligned}
       &\|\tilde{\varphi}_1\|_{\Omega_2}\preccurlyeq \|\varphi_1\|_{\Omega_1},\quad\|\tilde{\varphi}_1\|_{H^1(\Omega_2)}\preccurlyeq \|\varphi_1\|_{H^1(\Omega_1)},\\
       &\|\tilde{z}_1\|_{\Omega_2}\preccurlyeq \|z_1\|_{\Omega_1},\quad\|\tilde{z}_1\|_{H^1(\Omega_2)}\preccurlyeq \|z_1\|_{H^1(\Omega_1)}.
\end{aligned}
\end{equation}
On $\Omega_2$, applying Theorem \ref{regular0} (if $\varepsilon_2\geq\kappa$) or Theorem \ref{regular1} (if $\varepsilon_2<\kappa$) to $w=v|_{\Omega_2}-\nabla\tilde{\varphi}_1|_{\Omega_2}-\tilde{z}_1|_{\Omega_2}$ ($w\wedge n=0$ on $\partial\Omega_2$), we have $\varphi_2\in H_0^1(\Omega_2)$, $z_2\in [H_0^1(\Omega_2)]^d$ s.t. 
\begin{subequations}
\begin{align}
    &v|_{\Omega_2}-\nabla\tilde{\varphi}_1|_{\Omega_2}-\tilde{z}_1|_{\Omega_2}=\nabla\varphi_2\text{ + }z_2,\\
    &\|\kappa^\frac{1}{2}\varphi_2\|_{H^1(\Omega_2)}\text{ + }\|z_2\|_{W|_{\Omega_2}}\preccurlyeq\|v\|_{V|_{\Omega_2}}\text{ + }\|\kappa^\frac{1}{2}\nabla\tilde{\varphi}_1\|_{\Omega_2}\text{ + }\|\tilde{z}_1\|_{V|_{\Omega_2}}.\label{Omega2}
\end{align}
\end{subequations}
Here \eqref{Omega2} follows from similar reasons for \eqref{Omega1}.
Define $\varphi\in H_0^1(\Omega)$,  $z\in[H_0^1(\Omega)]^d$ as
\begin{equation*}
    \varphi:=\left\{\begin{aligned}
        &\varphi_1&&\text{ on }\Omega_1\\
        &\tilde{\varphi}_1\text{ + }\varphi_2&&\text{ on }\Omega_2
    \end{aligned}\right.,\quad z:=\left\{\begin{aligned}
        &z_1&&\text{ on }\Omega_1\\
        &\tilde{z}_1\text{ + }z_2&&\text{ on }\Omega_2
    \end{aligned}\right.,
\end{equation*}
and obtain $v=\nabla\varphi\text{ + }z$ on $\Omega$.
If $\varepsilon_1\geq\kappa$, it follows from \eqref{Omega2}, \eqref{extension}, \eqref{case1}, $\varepsilon_2\leq\varepsilon_1$ that
\begin{equation}\label{Omega21}
    \begin{aligned}
    &\|\kappa^\frac{1}{2}\varphi\|_{H^1(\Omega_2)}\text{ + }\|z\|_{W|_{\Omega_2}}\\
    &\preccurlyeq\|v\|_{V|_{\Omega_2}}\text{ + }\kappa^\frac{1}{2}\|\varphi_1\|_{H^1(\Omega_1)}\text{ + }(\kappa^\frac{1}{2}\text{ + }\varepsilon_2^\frac{1}{2})\|z_1\|_{\Omega_1}\text{ + }\varepsilon_2^\frac{1}{2}|z_1|_{H^1(\Omega_1)}\\
    &\preccurlyeq\|v\|_{V|_{\Omega_2}}\text{ + }\kappa^\frac{1}{2}\|v\|\text{ + }\varepsilon_1^\frac{1}{2}\|\nabla\times v\|_{\Omega_1}.
\end{aligned}
\end{equation}
Similarly when $\varepsilon_1<\kappa$, it follows from \eqref{Omega2}, \eqref{extension}, \eqref{case2}, $\varepsilon_2\leq\varepsilon_1<\kappa$ that
\begin{equation}\label{Omega22}
    \begin{aligned}
    \|\kappa^\frac{1}{2}\varphi\|_{H^1(\Omega_2)}\text{ + }\|z\|_{W|_{\Omega_2}}\preccurlyeq\|v\|_{V|_{\Omega_2}}\text{ + }\kappa^\frac{1}{2}\|v\|_{\Omega_1}.
\end{aligned}
\end{equation}
Combining \eqref{Omega1}, \eqref{Omega21}, \eqref{Omega22} completes the proof.
\end{proof}
\begin{remark}
The work \cite{XuZhu2011} gives a robust regular decomposition for the H(curl) interface problem with $\kappa=s\varepsilon$, $s\in(0,1]$. In contrast, Theorem \ref{regular2} is able to deal with large jump of $\varepsilon$ as well as large $\kappa\gg\varepsilon.$
\end{remark}

Given a Hilbert space $X$, let $X^\prime$ denote its dual space, and  $\langle\cdot,\cdot\rangle$ the action of $X^\prime$ on $X$. We introduce bounded linear operators
$A: V\rightarrow V^\prime$, $A_\Delta: H_0^1(\Omega)\rightarrow H^{-1}(\Omega)$, $A_W: W^d\rightarrow ([W]^d)^\prime$ as
\begin{align*}
    &\langle Av,w\rangle=(\varepsilon\nabla\times v,\nabla\times w)\text{ + }(\kappa v,w),\quad v,w\in V,\\
    &\langle A_\Delta v,w\rangle=(\nabla v,\nabla w)\text{ + }(v,w),\quad v,w\in H_0^1(\Omega),\\
    &\langle A_Wv,w\rangle=(\varepsilon\nabla v,\nabla w)\text{ + }(\kappa v,w),\quad v,w\in [W]^d.
\end{align*}
Let $r\in V^\prime$ be the residual given by 
\begin{equation}
    \langle r,v\rangle=(f,v)-(\varepsilon\nabla\times u_h,\nabla\times v)-(\kappa u_h,v),\quad v\in V.
\end{equation}
Clearly the inclusion $I: [W]^d\hookrightarrow V$ and the gradient operator $\nabla: W\rightarrow V$ are uniformly bounded w.r.t.~$\varepsilon$ and $\kappa.$ Then using such boundedness, Theorem \ref{regular2}, and the \emph{fictitious space lemma} (cf.~\cite{Nepomnyaschikh1992,HiptmairXu2007} and Corollary 5.1 in \cite{LiZikatanov2020arXiv}), we obtain the uniform spectral equivalence of two continuous operators 
\begin{equation}\label{equivalence}
    A^{-1}\simeq B:=\nabla (\kappa A_\Delta)^{-1}\nabla^\prime\text{ + }IA_W^{-1}I^\prime,
\end{equation}
where $I^\prime: V^\prime\rightarrow ([W]^d)^\prime$ and $\nabla^\prime: V^\prime\rightarrow W^\prime$ are adjoint operators. By $A^{-1}\simeq B$ from $V^\prime$ to $V$ in \eqref{equivalence} we mean $\langle R,A^{-1}R\rangle\simeq\langle R,BR\rangle,~\forall R\in V^\prime.$ It is noted that $A(u-u_h)=r\in V^\prime$. Therefore
a direct consequence of \eqref{equivalence} is
\begin{equation}\label{err}
\begin{aligned}
    &\|u-u_h\|^2_V=\langle A(u-u_h),u-u_h\rangle=\langle r,A^{-1}r\rangle\simeq\langle r,Br\rangle\\
    &=\langle \nabla^\prime r,(\kappa A_\Delta)^{-1}\nabla^\prime r\rangle\text{ + }\langle I^\prime r,A_W^{-1}I^\prime r\rangle=\kappa^{-1}\|\nabla^\prime r\|^2_{H^{-1}(\Omega)}\text{ + }\|I^\prime r\|^2_{([W]^d)^\prime}.
\end{aligned}
\end{equation}

\section{A Posteriori Error Estimates}
The goal of this paper is to derive a robust two-sided bound $\|u-u_h\|_V\simeq\eta_h.$
The quantity $\eta_h$ is computed from $u_h$ and split into element-wise error indicators for AMR. Such local error indicators are used to predict element errors in the current grid and mark those tetrahedra/hexahedra with large errors for subdivision.

When deriving the error estimator, we assume that the source $f$ is piecewise $H^1$-regular w.r.t.~$\mathcal{T}_h$. By $\mathcal{S}_h$ we denote the collection of $(d-1)$-simplexes in $\mathcal{T}_h$ that are not contained in $\partial\Omega$. Each $S\in\mathcal{S}_h$ shared by $T_S^\text{ + }, T_S^-\in\mathcal{T}_h$ is assigned with a unit normal $n_S$ pointing from $T_S^\text{ + }$ to $T_S^-$.
Let $h$, $h_s$ be the mesh size functions s.t.~$h|_T=h_T:=\text{diam}(T)$  $\forall T\in\mathcal{T}_h$, $h_s|_S=h_S:=\text{diam}(S)$  $\forall S\in\mathcal{S}_h$. 
The weighted mesh size functions are  \begin{equation*}
    \bar{h}:=\min\left\{\frac{h}{\sqrt{\varepsilon}},\frac{1}{\sqrt{\kappa}}\right\},\quad \bar{h}_s:=\min\left\{\frac{h_s}{\sqrt{\varepsilon_s}},\frac{1}{\sqrt{\kappa}}\right\},
\end{equation*}
where $\varepsilon_s|_S=\max\{\varepsilon_{T_S^\text{ + }},\varepsilon_{T_S^-}\}$ $\forall S\in\mathcal{S}_h$. For each $T\in\mathcal{T}_h$, $S\in\mathcal{S}_h,$ let $\Omega_T$ denote the union of elements in $\mathcal{T}_h$ sharing an edge with $T$, and $\Omega_S=\cup_{S\in\mathcal{S}_h, S\subset\partial T}\Omega_T$. For each $S\in\mathcal{S}_h,$ let $\llbracket\omega\rrbracket_S=\omega|_{T_S^\text{ + }}-\omega|_{T_S^-}$ be the jump of $\omega$ across $S.$  We define
\begin{equation*}
    \begin{aligned}
        &R_1|_T=-\nabla\cdot(f-\kappa u_h)|_T,\quad J_1|_S=\llbracket f-\kappa u_h\rrbracket_S\cdot n_S,\\
        &R_2|_T=(f-(\nabla\times)^*(\varepsilon\nabla\times u_h)-\kappa u_h)|_T,\quad J_2|_S=-\llbracket\varepsilon\nabla\times u_h\rrbracket_S\wedge n_S,
\end{aligned}
\end{equation*}
where $(\nabla\times)^*=\nabla\times$ in $\mathbb{R}^3$ and $(\nabla\times)^*=(-\partial_{x_2},\partial_{x_1})$ in $\mathbb{R}^2$.
By the element-wise Stokes' (in $\mathbb{R}^3$) or Green's (in $\mathbb{R}^2$) formula,
we have 
\begin{align}
        &\langle \nabla^\prime r,\psi\rangle=\langle r,\nabla\psi\rangle=\sum_{T\in\mathcal{T}_h}(R_1,\psi)_T\text{ + }\sum_{S\in\mathcal{S}_h}(J_1,\psi)_S,\quad\psi\in H^1_0(\Omega),\\
    &\langle I^\prime r,\phi\rangle=\langle r,\phi\rangle=\sum_{T\in\mathcal{T}_h}(R_2,\phi)_T\text{ + }\sum_{S\in\mathcal{S}_h}(J_2,\phi)_S,\quad\phi\in[W]^d.\label{Istarrexp}
\end{align}

In view of \eqref{err}, it remains to estimate $\|\nabla^\prime r\|_{H^{-1}(\Omega)}$ and $\|I^\prime r\|_{([W]^d)^\prime}.$ Let $(\cdot,\cdot)_{\mathcal{S}_h}$ denote the inner product $\sum_{S\in\mathcal{S}_h}(\cdot,\cdot)_S$ and $\|\cdot\|_{\mathcal{S}_h}$ the corresponding norm. Let $Q_h$ (resp.~$Q_h^s$) be the $L^2$ projection onto the space of discontinuous and piecewise polynomials of fixed degrees on $\mathcal{T}_h$ (resp.~$S_h$). The estimation of $\|\nabla^\prime r\|_{H^{-1}(\Omega)}$ is standard (cf.~\cite{LiZikatanov2020arXiv}) and given as
\begin{equation}\label{gradstarr}
    \|hR_1\|\text{ + }\|h_s^\frac{1}{2}J_1\|_{\mathcal{S}_h}-\text{osc}_{h,1}\preccurlyeq\|\nabla^\prime r\|_{H^{-1}(\Omega)}\preccurlyeq\|hR_1\|\text{ + }\|h_s^\frac{1}{2}J_1\|_{\mathcal{S}_h},
\end{equation}
where $\text{osc}_{h,1}:=\|h(R_1-Q_hR_1)\|\text{ + }\|h_s^\frac{1}{2}(J_1-Q^s_hJ_1)\|_{\mathcal{S}_h}$ is the data oscillation.
We also need the second data oscillation $\text{osc}_{h,2}:=\|\bar{h}(R_2-Q_hR_2)\|\text{ + }\|\bar{h}_s^\frac{1}{2}(J_2-Q^s_hJ_2)\|_{\mathcal{S}_h}.$
In the next lemma, we derive two-sided bounds for $\|I^\prime r\|_{([W]^d)^\prime}.$
\begin{lemma}\label{Istarr}
It holds that
\begin{equation*}
    \|\bar{h} R_2\|\text{ + }\|\varepsilon^{-\frac{1}{4}}\bar{h}_s^\frac{1}{2}J_2\|_{\mathcal{S}_h}-\emph{osc}_{h,2}\preccurlyeq\|I^\prime r\|_{([W]^d)^\prime}\preccurlyeq\|\bar{h} R_2\|\text{ + }\|\varepsilon^{-\frac{1}{4}}\bar{h}_s^\frac{1}{2}J_2\|_{\mathcal{S}_h}.
    \end{equation*}
\end{lemma}
\begin{proof}
The proof is similar to Lemma 4.4 of \cite{LiZikatanov2020arXiv} except the use of the modified Cl\'ement-type interpolation $\widetilde{\Pi}_h: [L^2(\Omega)]^d\rightarrow V_h^0$ proposed in \cite{CaiCao2015} for dealing with huge jump of $\varepsilon$. Here $V_h^0\subseteq V_h$ is the lowest order edge element space. For any $v\in[W]^d$ and $T\in\mathcal{T}_h,$
the analysis in Theorem 4.6 of \cite{CaiCao2015} implies that
\begin{align}
&\|v-\widetilde{\Pi}_hv\|_T\preccurlyeq h_T\varepsilon|_T^{-\frac{1}{2}}\|\varepsilon^\frac{1}{2}\nabla v\|_{\Omega_T}\leq h_T\varepsilon|_T^{-\frac{1}{2}}\|v\|_{W|_{\Omega_T}},\label{L2bound}\\
&\|\nabla(v-\widetilde{\Pi}_hv)\|_T\preccurlyeq\varepsilon|_T^{-\frac{1}{2}}\|\varepsilon^\frac{1}{2}\nabla v\|_{\Omega_T}\leq\varepsilon|_T^{-\frac{1}{2}}\|v\|_{W|_{\Omega_T}}.\label{H1bound}
\end{align}
The $L^2$-boundedness of $\widetilde{\Pi}_h$ implies that
\begin{equation}\label{kappabound}\|v-\widetilde{\Pi}_hv\|_T\preccurlyeq \|v\|_{\Omega_T}\leq\kappa^{-\frac{1}{2}}\|v\|_{W|_{\Omega_T}}.
\end{equation}
A direct consequence of \eqref{L2bound} and \eqref{kappabound} is
\begin{equation}\label{vT}
    \|v-\widetilde{\Pi}_hv\|_T\preccurlyeq \bar{h}_T\|v\|_{W|_{\Omega_T}}.
\end{equation}
Given a face/edge $S\in\mathcal{S}_h$, let $T$ be the element containing $S$ over which $\varepsilon$ is maximal. Using a trace inequality, \eqref{vT}, $h_S^{-1}\leq\bar{h}^{-1}_S\varepsilon_S^{-\frac{1}{2}},$ \eqref{H1bound}, $\bar{h}_S\simeq\bar{h}_T,$ we  have 
\begin{equation}\label{vS}
\begin{aligned}
    &\|v-\widetilde{\Pi}_hv\|^2_S\preccurlyeq h^{-1}_S\|v-\widetilde{\Pi}_hv\|_T^2\text{ + }\|v-\widetilde{\Pi}_hv\|_T\|\nabla (v-\widetilde{\Pi}_hv)\|_T\\
    &\preccurlyeq h_S^{-1}\bar{h}^2_T\|v\|^2_{W|_{\Omega_T}}\text{ + }\bar{h}_T\varepsilon|_T^{-\frac{1}{2}}\|v\|^2_{W|_{\Omega_T}}\preccurlyeq \varepsilon|_T^{-\frac{1}{2}}\bar{h}_S\|v\|^2_{W|_{\Omega_T}}.
\end{aligned}
\end{equation}
It follows from $r|_{V_h}=0,$ \eqref{Istarrexp}, the Cauchy--Schwarz inequality that
\begin{align*}
    &\|I^\prime r\|_{([W]^d)^\prime}=\sup_{v\in[W]^d,\|v\|_W=1}\langle r,v\rangle=\sup_{v\in [W]^d,\|v\|_W=1}\langle r,v-\widetilde{\Pi}_hv\rangle\\
    &\preccurlyeq\big(\|\bar{h}R_2\|\text{ + }\|\varepsilon_s^{-\frac{1}{4}}\bar{h}_s^\frac{1}{2}J_2\|_{\mathcal{S}_h}\big)\sup_{\substack{v\in [W]^d\\\|v\|_W=1}}\big(\|\bar{h}^{-1}(v-\widetilde{\Pi}_hv)\|\text{ + }\|\varepsilon_s^\frac{1}{4}\bar{h}_s^{-\frac{1}{2}}(v-\widetilde{\Pi}_hv)\|_{\mathcal{S}_h}\big).
\end{align*}
Then the upper bound of $\|I^\prime r\|_{([W]^d)^\prime}$ is a consequence of the above inequality and \eqref{vT}, \eqref{vS}. The uniform lower bound of $\|I^\prime r\|_{([W]^d)^\prime}$ w.r.t.~$\varepsilon, \kappa$ follows from the bubble function technique explained in~\cite{Verfurth2013} and extremal definitions of $\bar{h},$ $\bar{h}_s$, $\varepsilon_s$.
\end{proof}

For each $T\in\mathcal{T}_h$, we define the error indicator
\begin{equation*}
    \eta_h(T)=\kappa^{-1}h_T^2\|R_1\|_T^2\text{ + }\bar{h}|^2_T\| R_2\|^2_T\text{ + }\sum_{S\in\mathcal{S}_h, S\subset\partial T}\left\{\kappa^{-1}h_S\|J_1\|^2_S\text{ + }\bar{h}_s|_S\|\varepsilon^{-\frac{1}{4}}J_2\|^2_S\right\}.
\end{equation*}
Combining \eqref{err}, \eqref{gradstarr} and Lemma \ref{Istarr} leads to the robust a posteriori error estimate
\begin{equation}\label{robusteta}
    \sum_{T\in\mathcal{T}_h}\eta_h(T)-\text{osc}_{h,1}-\text{osc}_{h,2}\preccurlyeq\|u-u_h\|^2_V\preccurlyeq\sum_{T\in\mathcal{T}_h}\eta_h(T).
\end{equation}
\begin{remark}
Our analysis for \eqref{robusteta} is based on regular decomposition and minimal regularity  while the theoretical analysis of \emph{recovery} estimators in \cite{CaiCao2015} hinges on Helmholtz decomposition and full elliptic regularity of the underlying domain. Our estimator $\eta_h(T)$ is robust w.r.t.~both large jump of $\varepsilon$ and extreme magnitude of $\varepsilon,$ $\kappa$.
\end{remark}

\section{Numerical Demonstration of Robustness}
In the end, we focus on \eqref{Maxwell} with \emph{constant} and \emph{positive} $\varepsilon$ and
$\kappa$, which is a special case of the interface problem considered before. In this case, the error indicator $\eta_h(T)$ reduces to the one derived in \cite{LiZikatanov2020arXiv}. For constant $\varepsilon$ and $\kappa,$
the classical a posteriori error estimator for \eqref{disMaxwell} (cf.~\cite{BHHW2000,Schoberl2008}) reads  
\begin{equation*}
    \tilde{\eta}_h(T)=\kappa^{-1}h_T^2\|R_1\|_T^2\text{ + }\varepsilon^{-1}h_T^2\| R_2\|^2_T\text{ + }\sum_{S\in\mathcal{S}_h, S\subset\partial T}\left\{\kappa^{-1}h_S\|J_1\|^2_S\text{ + }\varepsilon^{-1}h_S\|J_2\|^2_S\right\}.
\end{equation*}
Although weighted with $\varepsilon$, $\kappa$, this estimator is not fully robust w.r.t.~$\varepsilon$ and  $\kappa$. In fact, the ratio $\|u-u_h\|_V/(\sum_{T\in\mathcal{T}_h}\tilde{\eta}_h(T))^\frac{1}{2}$ may tend to zero as $\varepsilon\ll\kappa$, i.e., the constant $\underline{C}$ in the lower bound $\underline{C}(\sum_{T\in\mathcal{T}_h}\tilde{\eta}_h(T))^\frac{1}{2}\leq\|u-u_h\|_V\text{ + }\text{h.o.t.}$ is not uniform. 

To validate the result,
we test $\eta_h(T)$ and $\tilde{\eta}_h(T)$ by the lowest order edge element discretization of \eqref{Maxwell} defined on $\Omega=[0,1]^2$ with the exact solution $u(x_1,x_2)=\big(\cos(\pi x_1)\sin(\pi x_2),\sin(\pi x_1)\cos(\pi x_2)\big).$
The initial partition of $\Omega$ is a $4\times4$ uniform triangular mesh. A sequence of nested grids is computed by uniform quad-refinement. Let $e=\|u-u_h\|_V,$ $\eta=(\sum_{T\in\mathcal{T}_h}\eta_h(T))^\frac{1}{2}$ and $\tilde{\eta}=(\sum_{T\in\mathcal{T}_h}\tilde{\eta}_h(T))^\frac{1}{2}$. Numerical results are shown in Table \ref{ConvergenceTab}. In its last row, we compute effectivity index ``eff'' of $\eta$ (resp.~$\tilde{\eta}$), which is the algorithmic mean of $e/\eta$ (resp.~$e/\tilde{\eta}$) over all grid levels. It is observed that the performance of $\eta$ is uniformly effective for all $\varepsilon, \kappa$, while the efficiency of $\tilde{\eta}$ deteriorates for small $\varepsilon$ and large $\kappa.$ 

\begin{table}[tbhp]
\caption{Convergence history of the lowest order edge element and error estimators}
\centering
\begin{tabular}{|c|c|c|c|c|c|c|c|c|c|}
\hline
number & $e $&$\eta$
 &$\tilde{\eta}$
&  $e$ &  $\eta$ &  $\tilde{\eta}$&$e$& $\eta$ & $\tilde{\eta}$ \\
 of &$\varepsilon=0.1$&$ \varepsilon=0.1$&$ \varepsilon=0.1$&$ \varepsilon=10^{-3}$&$ \varepsilon=10^{-3}$&$\varepsilon=10^{-3}$&$\varepsilon=10^{-5}$&$\varepsilon=10^{-5}$&$\varepsilon=10^{-5}$\\
 elements&$\kappa=10$&$ \kappa=10$&$ \kappa=10$&$ \kappa=10^3$&$ \kappa=10^3$&$\kappa=10^3$&$\kappa=10^5$&$\kappa=10^5$&$\kappa=10^5$\\
\hline

             32    &8.42e-1&3.72&3.94&8.24&3.72e+1&1.46e+3&8.24e+1&3.72e+2&1.46e+6\\
             128      &4.35e-1&2.04&2.04&4.30&2.04e+1&3.80e+2&4.30e+1&2.04e+2&3.80e+5\\
             512       &2.19e-1&1.04&1.04&2.18&1.06e+1&9.70e+1&2.18e+1&1.06e+2&9.64e+4\\
             2048      &1.10e-1&5.26e-1&5.26e-1&1.10&5.36&2.48e+1&1.10e+1&5.36e+1&2.42e+4\\
             8192      &5.49e-2&2.64e-1&2.64e-1&5.49e-1&2.69&6.61&5.49&2.69e+1&6.06e+3\\
\hline
eff & N/A &2.13e-1&2.11e-1 &N/A &2.09e-1 &3.33e-2&N/A &2.09e-1&3.51e-4\\
\hline
\end{tabular}
\label{ConvergenceTab}
\end{table}


\begin{thebibliography}{10}
\providecommand{\url}[1]{{#1}}
\providecommand{\urlprefix}{URL }
\expandafter\ifx\csname urlstyle\endcsname\relax
  \providecommand{\doi}[1]{DOI~\discretionary{}{}{}#1}\else
  \providecommand{\doi}{DOI~\discretionary{}{}{}\begingroup
  \urlstyle{rm}\Url}\fi

\bibitem{BHHW2000}
Beck, R., Hiptmair, R., Hoppe, R.H.W., Wohlmuth, B.: Residual based a
  posteriori error estimators for eddy current computation.
\newblock M2AN Math. Model. Numer. Anal. \textbf{34}(1), 159--182 (2000).


\bibitem{BernardiVerfurth2000}
Bernardi, C., Verf\"{u}rth, R.: Adaptive finite element methods for elliptic
  equations with non-smooth coefficients.
\newblock Numer. Math. \textbf{85}(4), 579--608 (2000).


\bibitem{CaiCao2015}
Cai, Z., Cao, S.: A recovery-based a posteriori error estimator for {${\bf
  H(curl)}$} interface problems.
\newblock Comput. Methods Appl. Mech. Engrg. \textbf{296}, 169--195 (2015).




\bibitem{HiptmairXu2007}
Hiptmair, R., Xu, J.: Nodal auxiliary space preconditioning in {${\bf H}({\bf
  curl})$} and {${\bf H}({\rm div})$} spaces.
\newblock SIAM J. Numer. Anal. \textbf{45}(6), 2483--2509 (2007).

\bibitem{HuShuZou2013}
Hu, Q., Shu, S., Zou, J.: A discrete weighted Helmholtz decomposition and its application.
\newblock Numer. Math. \textbf{125}, 153--189 (2013).


\bibitem{LiZikatanov2020CAMWA}
Li, Y., Zikatanov, L.: A posteriori error estimates of finite element methods
  by preconditioning.
\newblock Computers \& Mathematics with Applications.
\newblock \textbf{91}, 192--201 (2021).


\bibitem{LiZikatanov2020arXiv}
Li, Y., Zikatanov, L.: Nodal auxiliary a posteriori error estimates.
\newblock arXiv:2010.06774 (2020).


\bibitem{Nepomnyaschikh1992}
Nepomnyaschikh, S.V.: Decomposition and fictitious domains methods for elliptic
  boundary value problems.
\newblock In: Fifth {I}nternational {S}ymposium on {D}omain {D}ecomposition
  {M}ethods for {P}artial {D}ifferential {E}quations ({N}orfolk, {VA}, 1991),
  pp. 62--72. SIAM, Philadelphia, PA (1992)

\bibitem{Schoberl2008}
Sch\"{o}berl, J.: A posteriori error estimates for {M}axwell equations.
\newblock Math. Comp. \textbf{77}(262), 633--649 (2008).



\bibitem{Verfurth2013}
Verf\"{u}rth, R.: A posteriori error estimation techniques for finite element
  methods.
\newblock Numerical Mathematics and Scientific Computation. Oxford University
  Press, Oxford (2013).

\bibitem{Xu1992}
Xu, J.: Iterative methods by space decomposition and subspace correction.
\newblock SIAM Rev. \textbf{34}(4), 581--613 (1992).

\bibitem{XuZhu2011}
Xu, J., Zhu, Y.: Robust preconditioner for {${\bf H}(\bf{curl})$} interface
  problems.
\newblock In: Domain decomposition methods in science and engineering {XIX},
  \emph{Lect. Notes Comput. Sci. Eng.}, vol.~78, pp. 173--180. Springer,
  Heidelberg (2011).
\end{thebibliography}
\end{document}